\pdfoutput=1
\documentclass[11pt]{amsart}

\usepackage{amssymb,amsfonts}
\usepackage{amsmath,amscd}
\usepackage[all,arc]{xy}
\usepackage{enumerate}
\usepackage{mathrsfs}
\usepackage[pdftex]{graphicx}
\usepackage[utf8]{inputenc}
\usepackage[T1]{fontenc}
\usepackage[usenames,dvipsnames]{color}
\usepackage[bookmarks=false]{hyperref}
\usepackage{dsfont}
\usepackage{tikz}
\usetikzlibrary{automata,positioning}

\theoremstyle{plain}
\newtheorem{thm}{Theorem}[section]
\newtheorem{cor}[thm]{Corollary}
\newtheorem{prop}[thm]{Proposition}
\newtheorem{lem}[thm]{Lemma}
\newtheorem{conj}[thm]{Conjecture}

\newenvironment{customThm}[1]{\paragraph*{\textbf{Theorem #1.}}\itshape}{\par}
\newenvironment{customThm1}[2]{\paragraph*{\textbf{Theorem #1} (#2)\textbf{.}}\itshape}{\par}

\theoremstyle{definition}
\newtheorem{defn}[thm]{Definition}

\newtheorem{example}[thm]{Example}

\theoremstyle{remark}
\newtheorem{rmk}[thm]{Remark}

\newcommand{\bbC}{\mathbb{C}} 

\newcommand{\bbN}{\mathbb{N}} 

\newcommand{\bbR}{\mathbb{R}} 

\newcommand{\bbZ}{\mathbb{Z}} 
\newcommand*{\defeq}{\mathrel{\vcenter{\baselineskip0.5ex \lineskiplimit0pt
			\hbox{\scriptsize.}\hbox{\scriptsize.}}}%
	=}

\makeatletter
\let\c@equation\c@thm
\makeatother
\numberwithin{equation}{section}

\title{Basmajian-type identities and Hausdorff dimension of limit sets}

\author{Yan Mary He}
\address{Department of Mathematics\\
  University of Chicago\\
  Chicago, IL 60637}
\email{he@math.uchicago.edu}
\date{\today}

\begin{document}

\begin{abstract} 
In this paper, we study Basmajian-type series identities on holomorphic families of Cantor sets associated to one-dimensional complex dynamical systems. We show that the series is absolutely summable if and only if the Hausdorff dimension of the Cantor set is strictly less than one. Throughout the domain of convergence, these identities can be analytically continued and they exhibit nontrivial monodromy.
\end{abstract}

\maketitle

\section{Introduction}
If $\mathcal{C} \subset [0,1]$ is a Cantor set of zero measure, the Hausdorff
dimension of $\mathcal{C}$ is the limit of $\log{a_n}/\log{n}$, where $a_n$ is the
length of the $n$th biggest component of $[0,1]-\mathcal{C}$. There is no obvious analog of this theorem for an arbitrary Cantor set $\mathcal{C} \subset \bbC$, but
for a {\em family} of Cantor sets $\mathcal{C}_z \subset \bbC$ depending 
holomorphically on a complex parameter $z$, we can sometimes obtain such
a relation.

Note that if the measure of $\mathcal{C} \subset [0,1]$ is zero, then $a_n$ satisfy an identity $$1 = \sum_{n=1}^{\infty} a_n.$$ For $\mathcal{C}_z \subset \bbC$ depending on $z$, we obtain by analytic continuation a 
formal holomorphic {\em family} of identities $$S(z) = \sum_{n=1}^{\infty} a_n(z).$$ Thus it is natural to investigate the conditions under which the right hand side is absolutely summable. 

In this paper, we study holomorphic families of Cantor sets $\mathcal{C}_z$ associated 
to familiar $1$-dimensional complex dynamical systems, and for such families we introduce identities of this form (which have a natural geometric interpretation when $\mathcal{C}_z \subset \bbR$) and show that the right hand sides are absolutely summable if and only if the Hausdorff dimension of $\mathcal{C}_z$ is strictly less than $1$. 

The identities themselves are of independent interest --- even in the case
of $\mathcal{C} \subset \bbR$, where a special case is Basmajian's orthospectrum
identity for a hyperbolic surface.

\subsection{Complexified Basmajian's identity}
If $\Sigma$ is a compact hyperbolic surface with geodesic boundary, an
{\em orthogeodesic} $\gamma \subset \Sigma$ is a properly immersed geodesic
arc perpendicular to $\partial \Sigma$. Basmajian (\cite{Bas})
proved the following identity:
\begin{equation}\label{bas_torus}
	\text{length}(\partial \Sigma) = \sum_{\gamma} 2 \log \coth \left(\dfrac{\text{length}(\gamma)}{2}\right)
\end{equation}
where the sum is taken over all orthogeodesics $\gamma$ in $\Sigma$.
The geometric meaning of this identity is that there is a canonical decomposition of 
$\partial \Sigma$ into a Cantor set (of zero measure), plus a countable
collection of complementary intervals, one for each orthogeodesic, whose
length depends only on the length of the corresponding orthogeodesic.

For simplicity, one can look at surfaces with a single boundary component.
A hyperbolic structure on $\Sigma$ is the same as a discrete faithful
representation $\rho:\pi_1(\Sigma) \to \text{PSL}(2,\bbR)$ acting on the upper
half-plane model in the usual way. After conjugation, we may
assume that $\rho(\partial \Sigma)$ stabilizes the positive imaginary axis 
$\ell$. 

Orthogeodesics correspond to double cosets of $\pi_1(\partial \Sigma)$ in
$\pi_1(\Sigma)$. For each nontrivial double coset $\pi_1(\partial \Sigma) \alpha \pi_1(\partial \Sigma)$, the hyperbolic geodesic $\alpha(\ell)$ corresponds to another boundary component of $\widetilde{\Sigma}$,
and the contribution to Basmajian's identity from this term is
$\log(\rho(\alpha)(0)/\rho(\alpha)(\infty))$; hence
$$\text{length}({\partial \Sigma}) = \sum_\alpha \log(\rho(\alpha)(0)/\rho(\alpha)(\infty)).$$

If we deform $\rho$ to some nearby representation $\rho_z:\pi_1(\Sigma) \to \text{PSL}(2,\bbC)$,
and replace each $\rho$ by $\rho_z$ above, we obtain a formula for the
{\em complex length} of $\rho_z(\partial \Sigma)$. This is the desired
complexification of Basmajian's identity.

The identity makes sense, and is absolutely convergent, exactly throughout
the subset $\mathcal{S}_{<1}$ of Schottky space $\mathcal{S}$ where the Hausdorff dimension of
the {\em limit set} of $\rho_z(\pi_1(\Sigma))$ is less than one. Since $\log$ is
multivalued, it is important to choose the correct branch at a real 
representation, and then follow the branch by analytic continuation. The space $\mathcal{S}_{<1}$ is not simply-connected and some terms exhibit monodromy (in the form of integer multiples of $2\pi i$) when they are analytically continued around homotopically nontrivial loops.

\subsection{Quadratic polynomials}
If the complex parameter $c$ is less than $-2$, the Julia set $J_c$ of the quadratic polynomial $f_c(z) = z^2+c$ is a Cantor set of zero measure contained in the real line. Thus there is a natural Basmajian-type identity associated to $J_c$. If we perturb $c$ off the real line in the complement of the Mandelbrot set $\mathcal{M}$, we obtain a formal family
of complexified identities, depending holomorphically on $c$. Analogous to the case of Basmajian's identity, the complexified identity holds exactly on the subset $\left(\bbC \setminus \mathcal{M}\right)_{<1}$ of $\bbC \setminus \mathcal{M}$ where the Hausdorff dimension of the Julia set is strictly less than one. 

$\pi_1(\bbC \setminus \mathcal{M}) = \bbZ$ and the generator induces new monodromy terms in the series identity. Our criterion for convergence gives us a method to numerically compute the locus in $\bbC \setminus \mathcal{M}$ where the Hausdorff dimension of $J_c$ is equal to $1$. Experimentally this locus appears to be a topological circle with a ``cusp'' at $-2$.

\subsection{Thermodynamic formalism}
Our main analytic result --- the relation between the growth rate of the terms in the right hand side of the identities and the Hausdorff dimension --- is proved by using the Thermodynamic Formalism developed by Ruelle, Bowen and others in the 1970's.

The main geometric idea is that the spectrum of the transfer operator is controlled by the geometric contraction rate of sets in a suitable Markov partition, and this in turn can be related to the size of the terms in the identity, precisely because the dynamical systems are conformal.

\subsection{Relations to $L$-functions}
Dirichlet's unit theorem expresses certain covolumes (of units in an algebraic number field) in terms of special values of $L$-functions, which have a series decompostion, of which the Riemann zeta function is the simplest example. Basmajian's identity expresses in a similar way a (co)volume as a series, expressed over topological terms. We suggest that the study of our families of Basmajian-type identities is analogous to the idea of studying $L$-functions expressed in series form.

\subsection{Statement of results}
In section 2, we consider an elementary example of our general theory, namely the case of an iterated function system generated by a pair of planar similarities.

In section 3, we study complexified Basmajian's identity for Schottky groups and exhibit loops in $\mathcal{S}_{<1}$ with nontrivial monodromy. The main theorem of this section is the following:
\vspace{0.3cm}
\begin{customThm1}{3.10}{Complexified Basmajian's identity}
Suppose $\rho_0: F_n \to \text{PSL}(2, \bbC)$ is a Fuchsian marking corresponding to a hyperbolic surface $\Sigma$ with geodesic boundaries $a_1,\cdots, a_k$. Let $\alpha_1,\cdots, \alpha_k \in \pi_1M$ represent the free homotopy classes of $a_1,\cdots, a_k$. If $\rho$ is in the same path component as $\rho_0$ in $\mathcal{S}_{<1}$, then
\begin{equation}\label{Intro_cx_Bas_id}
\sum_{j=1}^k l(\rho(\alpha_j)) = \sum_{p,q =1}^k \sum_{w \in \mathcal{L}_{p,q}} \log [\alpha_p^+,\alpha_p^-;w\cdot \alpha_q^+,w \cdot \alpha_q^-] \text{~~mod~~} 2\pi i,
\end{equation}
where $\alpha_j^+,\alpha_j^-$ are the attracting and repelling fixed points of $\rho(\alpha_j)$, respectively. Moreover, the series converges absolutely.
\end{customThm1}
\vspace{0.3cm}
Here $\mathcal{L}_{p,q}$ is a set of double coset representatives associated to boundary components $a_p$ and $a_q$. The key ingredient in the proof of the theorem is the following analytic result, whose proof is deferred until section 6.
\vspace{0.3cm}
\begin{customThm}{3.8}
	Given a marked Schottky representation $\rho:F_n \to \text{PSL}(2,\bbC)$, the infinite series (\ref*{Intro_cx_Bas_id}) converges absolutely if and only if the Hausdorff dimension of the limit set $\Lambda_{\Gamma}$ of the Schottky group $\Gamma=\rho(F_n)$ is strictly less than one.
\end{customThm}
\vspace{0.3cm}

In section 4 we study Basmajian-type identities for quadratic polynomials and numerically plot the Hausdorff dimension one locus in $\bbC \setminus \mathcal{M}$. Our main results are the following:
\vspace{0.3cm}
\begin{customThm}{4.4}
	For complex parameter $c \in \left(\bbC \setminus \mathcal{M}\right)_{<1}$, let $T_1$ and $T_2$ be the two branches of $f_c^{-1}$ and $z_1$ be the fixed point of $T_1$, then the following identity holds
	\begin{equation} \label{Intro_Julia_id}
		z_1 - (-z_1) = \sum_{w \in \{T_1,T_2\}^*} (-1)^\eta \Big( w(T_1(-z_1)) - w(T_2(-z_1)) \Big),
	\end{equation}
	where $\eta$ is the number of $T_2$'s in the word $w$.
\end{customThm}
\vspace{0.3cm}

Again, we will state the convergence theorem and defer its proof until section 6.
\vspace{0.3cm}
\begin{customThm}{4.3}
	The series in \ref{Intro_Julia_id} is absolutely convergent if and only if the Hausdorff dimension of $J_c$ is strictly less than one.
\end{customThm}
\vspace{0.3cm}

The monodromy action of a nontrival loop in $\left(\bbC \setminus \mathcal{M}\right)_{<1}$ on $J_c$ is a map $\phi$, which, under the symbolic coding, simply exchanges the labels $1$ and $2$ and hence induces a new identity on $J_c$:
\begin{equation} 
z_2 - (-z_2) = \sum_{w \in \{T_1,T_2\}^*} (-1)^{\eta'} \Big( w(T_2(-z_2)) - w(T_1(-z_2)) \Big)
\end{equation}
where $\eta'$ is the number of $T_1$'s in the word $w$.

In section 5 we review elements of the Thermodynamic Formalism and introduce the main technical tools we will need.

Finally, in section 6 we prove Theorems \ref*{thm_cxBas_conv} and \ref*{thm_Julia_conv}.

\subsection{Acknowledgements}
I would like to thank Danny Calegari for posing the problems studied in this paper and for many helpful conversations and insightful suggestions. I thank Peter Shalen for many helpful conversations and a careful reading. I am also grateful to Sarah Koch, Alden Walker and Quoc Ho for useful conversations regarding complex dynamics and programming.

\section{Some linear examples}\label{sec_linear}
In this section, we give the simplest example to illustrate our main theorem, that of a conformal iterated function system in $\bbC$ generated by two similarities.

If $T_1,\cdots, T_k : \bbR^n \to \bbR^n$ is a finite family of contractions, the {\it limit set} is the set of accumulation points of the semigroup generated by $\{T_i\}$. These semigroups are usually called iterated function systems or IFS for short.

Dimensions of limit sets of IFSs have been widely studied. Among several notions of dimension, we will concentrate on box dimension and Hausdorff dimension. Conveniently, for the examples we study these two notions of dimension agree. We denote Hausdorff dimension by $\text{dim}_\text{H}$ and box dimension by $\text{dim}_\text{B}$. For a definition, see \cite{Fal}.

For any compact set $E \subset \bbR ^n$, there is an inequality $\text{dim}_\text{H} E \le \text{dim}_\text{B} E$. However, for limit sets of conformal IFSs, there is an important dynamical condition under which the Hausdorff dimension equals the box dimension.

\begin{defn} [Open set condition]
The family of maps $T_1,\cdots,T_k: \bbR^n \to \bbR^n$ satisfy the {\it open set condition} if there exists an open set $V \subset \bbR^n$ such that $T_i(V)$ are contained in $V$ for all $i=1,\cdots,k$ and $T_i(V) \cap T_j(V) = \emptyset$ for all $i \neq j$.
\end{defn}

\begin{prop}[\cite{Pol}, Proposition 2.1.4]
For a conformal IFS satisfying the open set condition with limit set $\Lambda$, $\text{dim}_\text{H} \Lambda = \text{dim}_\text{B} \Lambda$. 
\end{prop}

\subsection{Middle-third Cantor set} 
The middle-third Cantor set $\mathcal{C}$ can be viewed as a ``cut-out set'' -- a set obtained from the interval $I=[0,1]$ by cutting out a sequence of disjoint intervals $(\frac{1}{3}, \frac{2}{3}), (\frac{1}{9}, \frac{2}{9}), (\frac{7}{9}, \frac{8}{9})\cdots$
Let $a_i$ denote the length of the $i$th interval, then $$a_i = (1/3)^{(\left \lfloor{\log_2i}\right \rfloor + 1)}, \text{ where} \left \lfloor{x}\right \rfloor \text{is the floor function.}$$ 
By construction, it is obvious that the length of the total interval $[0,1]$ equals the infinite sum of lengths of complimentary intervals; namely, we have the following simplest Basmajian-type identity:
\begin{equation} \label{Cantor_id}
\left\vert{[0,1]}\right\vert = 1 = \sum_{i=1}^{\infty}a_i = \frac{1}{3} + \left(\frac{1}{9}+\frac{1}{9}\right) + \left(\frac{1}{27}+\frac{1}{27}+\frac{1}{27}+\frac{1}{27}\right) + \cdots 
\end{equation}

On the other hand, $\mathcal{C}$ is also the limit set of the conformal IFS $\{T_1, T_2: \bbR \to \bbR\}$ given by $$T_1(x) = \frac{x}{3} \text{ and } T_2(x) = \frac{x}{3} + \frac{2}{3}.$$ 
Note that $0$ is the fixed point of $T_1$, $1$ is the fixed point of $T_2$ and all other points in $\mathcal{C}$ are the images of $0$ and $1$ under iterations of $T_1$ and $T_2$.
In terms of $T_1$, $T_2$ and their fixed points, identity \ref{Cantor_id} becomes
\begin{equation} \label{dyn_Cantor_id}
\lvert z_1 - z_2\rvert = \sum_{w \in \{T_1,T_2\}^{*}} \lvert w(T_2(z_1)) - w(T_1(z_2))\rvert
\end{equation}
where $z_1$ and $z_2$ are the fixed points of $T_1$ and $T_2$ respectively and the summation is taken over all the words $w$ in the {\it alphabet} $\{T_1, T_2\}$.

The following theorem due to Falconer shows that the box dimension of a real one-dimensional cut-out set (and in particular $\text{dim}_{\text{H}}\mathcal{C}$) is controlled by the asymptotic sizes of complementary intervals.

\begin{thm} [Falconer \cite{Fal}] \label{prop_Falc}
Let $A$ be a compact subset of $\bbR$ and let $\{A_i\}_{i = 1}^{\infty}$ be a sequence of disjoint open subintervals of $A$ such that if $a_i = |A_i|$ with $a_1 \ge a_2 \ge a_3 \ge \cdots$, then $|A| = \sum_{i = 1}^{\infty} a_i$. Write $E = A \setminus (\bigcup_{i=1}^{\infty} A_i)$. Let $$a = -\liminf \dfrac{\log a_n}{\log n} \text{ and } b = -\limsup \dfrac{\log a_n}{\log n}.$$ Then $$\frac{1}{a} \le \underline{\text{dim}}_\text{B} E \le \overline{\text{dim}}_\text{B} E \le \frac{1}{b}.$$
\end{thm}

\begin{example}
For the middle-third Cantor set $\mathcal{C}$, $\text{dim}_\text{H}\mathcal{C} =\text{dim}_\text{B}\mathcal{C} = \log_23 = \displaystyle\lim_{n\to \infty} \dfrac{\log a_n}{\log n}$.
\end{example}

\subsection{Semigroups of similarities} 
Consider the following conformal IFS of two complex similarities $\{f_c, g_c : \bbC \to \bbC\}$ given by $$f_c(z) = cz \text{~~and~~} g_c(z) = c(z-1)+1$$ where $0<|c|<1$ is a complex parameter. The limit set $\Lambda_c$ is either connected or a Cantor set (\cite{C-K-W}, Lemma 5.2.1). 

The middle-third Cantor set corresponds to the case $c=1/3$. Hence, when $\Lambda_c$ is a Cantor set, formula \ref{dyn_Cantor_id} gives a formal Basmajian-type identity
\begin{equation} \label{geo_id}
1 = \sum_{n=0}^{\infty} (2c)^n(1-2c)
\end{equation}
which is just a geometric series. It is (absolutely) convergent if and only if $|c| < 1/2$. 

On the other hand, this conformal IFS satisfies the open set condition, and therefore by Moran's theorem, the Hausdorff dimension of $\Lambda_c$ is given by  
$$\text{dim}_{\text{H}}\Lambda_c = -\log 2/ \log |c|,$$ and $\text{dim}_{\text{H}}\Lambda_c<1$ if and only if $\lvert c \rvert < 1/2$. 

Hence, we see that the series in \ref{geo_id} is absolutely convergent if and only if
$\text{dim}_{\text{H}}\Lambda_c<1$.

\section{Schottky Groups: the Complexified Basmajian Identity} \label{sec_cxBas}
Our main goal for this section is to extend Basmajian's identity \ref{bas_torus} to a suitable subspace of marked Schottky space $\mathcal{S}$ via analytic continuation. There is a ``formal'' identity for any marked Schottky group $\Gamma$. However, the interpretation is problematic unless one deals with the issue of convergence. As in the example of section 2, it turns out that the series is absolutely summable if and only if the limit set $\Lambda_{\Gamma}$ has Hausdorff dimension strictly less than one. Denote by $\mathcal{S}_{<1}$ the space of Schottky groups whose limit set has Hausdorff dimension strictly less than one. Then the main theorem of this section (Theorem \ref{thm_cx_Bas}) states that this is the maximal domain on which the extended Basmajian's identity holds.

On the other hand, the extended Basmajian's identity may serve as a tool to study the topology of $\mathcal{S}_{<1}$. More specifically, when analytically continued along a loop in $\mathcal{S}_{<1}$, finitely many terms in the series will exhibit monodromy --- their imaginary part changes by integer multiples of $2\pi$. We will present examples of such loops.

\subsection{Marked Schottky space}
Let $F_n = \langle g_1,\cdots, g_n \rangle$ be a free group on $n$ generators. A discrete faithful representation $\rho : F_n \to \text{PSL}(2,\bbC)$ is called a {\it marked Schottky representation} if there is a subsurface $E \subset \hat{\bbC}$ (homeomorphic to a sphere with $2n$ holes) with boundary components $C_i$, $C_i'$ for $i=1,\cdots,n$ so that $\rho(g_i)(C_i) = C_i'$ and $\rho(g_i)(E) \cap E = C_i'$. The representation is called a {\it marked Fuchsian representation} or a {\it Fuchsian marking} if it is conjugate to a representation into $\text{PSL}(2,\bbR)$. Two marked Schottky representations $\rho_1$ and $\rho_2$ are {\it equivalent} if they are conjugate. The space of equivalence classes of marked Schottky representations is called the {\it marked Schottky space}, denoted by $\mathcal{S}$. A standard reference is \cite{Bers}.

The image of $F_n$ under a Schottky representation is a {\it Schottky group} $\Gamma$. It is well-known that every nontrivial element in $\Gamma$ is loxodromic and the limit set $\Lambda_{\Gamma}$ is a Cantor set.

$\mathcal{S}$ can be parametrized by the fixed points and the trace squares of the $\rho(g_i)$'s. Denote Fix$^-$ and Fix$^+$ the repelling and attracting fixed points, respectively. Normalize the representation so that $\text{Fix}^-\rho(g_1)=0$, $\text{Fix}^+\rho(g_1)= \infty$ and $\text{Fix}^-\rho(g_2)=1$. Then $\rho$ is uniquely determined by
$$\left(\text{Fix}^+\rho(g_2), \text{Fix}^-\rho(g_3), \text{Fix}^+\rho(g_3),\cdots, \text{Fix}^+\rho(g_n); \text{tr}^2(\rho(g_1)),\cdots, \text{tr}^2(\rho(g_n)) \right)$$ which gives a parametrization of $\mathcal{S}$ as a subspace of $\hat{\bbC}^{2n-3} \times \bbC^{2n}$. It is a standard fact that $\mathcal{S}$ is open and connected. Moreover,
$\mathcal{S}_{<1}$ is also open because Hausdorff dimension is an analytic function on the deformation space of Schottky groups (see Corollary \ref{cor_ana_hdim}).

In the next three subsections, we reformulate both sides of Basmajian's identity in terms of representations.

\subsection{Complex length}
For $A \in \text{PSL}(2,\bbC)$, the {\it complex length} of $A$ is 
$$l(A) = \cosh^{-1}\left(\dfrac{\text{tr~}A^2}{2}\right).$$ 
Note that complex length is only defined up to multiples of $2\pi i$.

Thus, for any $g \in F_n$, the complex length $l(\rho(g))$ is an analytic function of the coordinates on $\mathcal{S}$. The next lemma states that the real part of complex length stays positive when we deform a Fuchsian marking.

\begin{lem}
	Let $\rho_0$ be a Fuchsian marking with $l(\rho(g))$ real positive for all $g \in F_n$. Then Re$(l(\rho_t(g)))>0$ for all $g \in F_n$ when analytically continued along a path $\rho_t$ in $\mathcal{S}_{<1}$.
\end{lem}
\begin{proof}
	If there were some $t$ and some $g \in F_n$ such that Re$(l(\rho_t(g)))=0$, then it would contradict the fact that every element in a Schottky group is loxodromic.
\end{proof}

\subsection{Double cosets, finite state automata and orthogeodesics}
In this subsection, we show that orthogeodesics on $\Sigma$ correspond to certain double cosets of $\pi_1 \Sigma$. There is an efficient coset enumeration algorithm which can be implemented for numerical calculations. Looking ahead, this enumeration parallels the encoding of the dynamical systems as a subshift of finite type, a step in the application of the Thermaldynamic Formalism carried out in section 6.

Let $\alpha_1,\cdots, \alpha_k \in \pi_1\Sigma$ represent the free homotopy classes of boundary geodesics $a_1,\cdots, a_k$, respectively. Denote $H_j \defeq \langle \alpha_j \rangle$, the subgroup of $\pi_1\Sigma$ generated by $\alpha_j$. Let $\mathcal{DC}(\Sigma)$ denote the set of double cosets of the form $H_pwH_q$ where $w \in \pi_1\Sigma$ is not in $H_p \cap H_q$, for $p,q =1,\cdots,k$. Denote by $[w]_{p,q}$ the class of the double coset $H_p w H_q$.

\begin{prop}\label{prop_isom}
There is a bijection $$\Phi : \mathcal{DC}(\Sigma) \to \{\text{Orthogeodesics on } \Sigma\}.$$
\end{prop}
\begin{proof}
Every homotopically nontrivial proper arc has a unique orthogeodesic in its homotopy class (rel. endpoints), which can be seen e.g. by curve shortening.
\end{proof}

Let $S$ be a symmetric generating set of $\pi_1\Sigma$. Choose a total order on $S$. This determines a unique reduced lexicographically first (RedLex) representative $w$ of each double coset $H_p w H_q$. Let $\mathscr{L}_{p,q}$ be the set of nontrivial RedLex double coset representatives for fixed $p, q$. Then the set $$\mathscr{L} \defeq \coprod\limits_{1 \le p, q \le k} \mathscr{L}_{p,q}$$ is naturally in bijection with the set of orthogeodesics on $M$.

\begin{defn}
	A {\it finite state automaton} on a fixed alphabet $S$ is a finite directed graph $G$ with a starting vertex $*$ and a subset of the vertices called the {\it accept states}, whose oriented edges are labeled by letters of $S$ so that there is at most one outgoing edge with any given label at each vertex. A word is {\it accepted} by a finite state automaton if there is a path realizing the word which starts with $*$ and ends on an accept state.
\end{defn}

For a fixed finite alphabet $\mathcal{A}$, a {\it language} is a subset of the set of all words on $\mathcal{A}$. A language is {\it regular} if it consists of exactly the words accepted by some finte state automaton.

\begin{prop}
	$\mathscr{L}$ is a {\it regular language} over the {\it alphabet} $S$.
\end{prop}
\begin{proof}
	Note that $\pi_1\Sigma$ is hyperbolic and $H_j$'s are quasiconvex. Then by $\delta$-thinness of triangles of a hyperbolic group, the language of all shortest coset representatives is regular. Also, for a given coset, all shortest coset representatives synchronously fellow-travel. Thus, $\mathscr{L}$ is regular.
\end{proof}

The finite state automaton parametrization of a regular language gives rise to a fast coset enumeration algorithm, as demonstrated in the following example.

\begin{example}(Torus with a geodesic boundary) \label{ex_torus}
	Let $\Sigma$ be a torus with a single geodesic boundary and let $S = \{a, b, A, B\}$ be a symmetric generating set for $\pi_1\Sigma$, where capital letters denote inverses. Order $S$ by $a>b>A>B$. Then (the free homotopy class of) the boundary geodesic is represented by the commutator $[a,b]$. Let $H \defeq \langle [a,b]\rangle$. By Proposition \ref{prop_isom}, the set of orthogeodesics on $\Sigma$ is parametrized by the set $\mathscr{L}$ of nontrivial RedLex representatives $w$ of $HwH$.
	
	More precisely, $w \in \mathscr{L}$ if and only if it satisfies the following conditions:
	\begin{enumerate}
		\item $w$ is reduced, i.e. $w$ does not contain $aA$, $Aa$, $bB$ or $Bb$;
		\item $w$ does not start with $abA$ or $ba$;
		\item $w$ does not end with $BA$ or $bAB$.
	\end{enumerate} 
	
	Figure \ref{fig:automaton} is a finite state automaton parametrizing the regular language $\mathscr{L}$. In this automaton, double-circled-nodes are accept states and red nodes are reject states. Accepted words are obtained by starting at $*$, going along edges and ending at an accept state and the word is the concatenation of edge labels.
	
	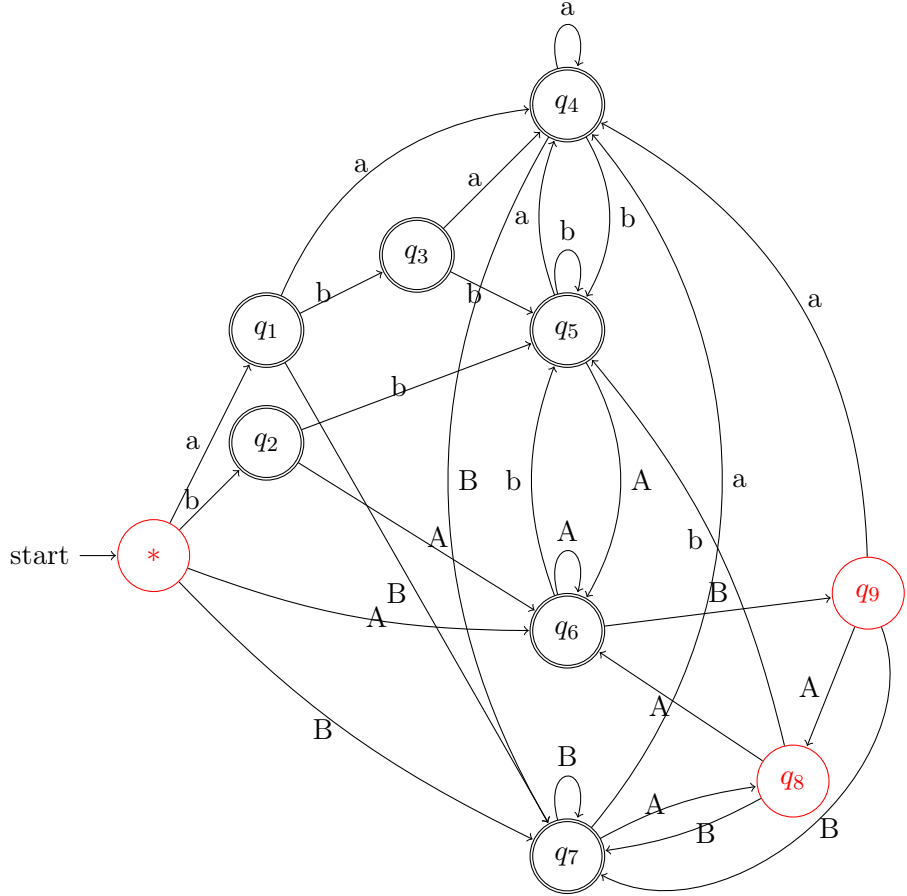
\begin{figure}[h!]
		\centering
		\begin{tikzpicture} [shorten >=0.5pt,node distance=0.7cm,auto] 
		\node[red][state, initial] (q0) at (0.5,0)  {$*$}; 
		\node[state,accepting] (q1) at (2,3) {$q_1$}; 
		\node[state,accepting] (q2) at (2,1.5) {$q_2$};
		\node[state,accepting] (q3) at (4,4) {$q_3$}; 
		\node[state,accepting] (q4) at (6,6) {$q_4$};
		\node[state,accepting] (q5) at (6,3) {$q_5$}; 
		\node[state,accepting] (q6) at (6,-1) {$q_6$};
		\node[state,accepting] (q7) at (6,-4) {$q_7$};  
		\node[red,state] (q8) at (9,-3) {$q_8$};
		\node[red,state] (q9) at (10,-0.5) {$q_9$};  
		\path[->] 
		(q0) edge [left]  node {a} (q1)
		edge [left]  node [swap] {b} (q2)
		edge [bend right=10,right]  node [swap] {A} (q6)
		edge [bend right=10,left]  node [swap] {B} (q7)
		(q1) edge [left]  node  {b} (q3)
		edge [bend left, left] node [swap] {a} (q4)
		edge [left] node [swap] {B} (q7)
		(q2) edge [left]  node [swap] {b} (q5) 
		edge [right] node {A} (q6)
		(q3) edge [left]  node [swap] {a} (q4) 
		edge [left] node {b} (q5)
		(q4) edge [loop above]  node [swap] {a} (q4) 
		edge [bend left] node {b} (q5)
		edge [bend right = 30, right] node {B} (q7)
		(q5) edge [loop above] node [swap] {b} (q5) 
		edge [bend left=20] node {a} (q4)
		edge [bend left=30] node {A} (q6)
		(q6) edge [loop above] node [swap] {A} (q6) 
		edge [bend left=20] node {b} (q5)
		edge [above] node {B} (q9)
		(q7) edge [loop above] node [swap] {B} (q7) 
		edge [bend left=10, left] node {A} (q8)
		edge [bend right = 40, right] node {a} (q4)
		(q8) edge [bend right=13, left] node [swap] {b} (q5) 
		edge [left] node {A} (q6)
		edge [bend left=10, right] node {B} (q7)
		(q9) edge [bend right=30, right] node [swap] {a} (q4) 
		edge [left] node {A} (q8)
		edge [bend left=70, right] node {B} (q7);
		\end{tikzpicture}
		\caption{A finite state automaton for torus with geodesic boundary} 
		\label{fig:automaton}
	\end{figure}
\end{example}

\subsection{Cross ratios}
Let $\rho_0: F_n \to \text{PSL}(2,\bbR)$ be the Fuchsian marking corresponding to a hyperbolic structure on $\Sigma$. Suppose $\gamma_{p,q}$ is an orthogeodesic on $\Sigma$ running from $a_p$ to $a_q$ with RedLex double coset representative $w$. Then its hyperbolic length $\text{length}(\gamma_{p,q})$ satisfies 

\begin{equation}\label{coth}
\coth\left(\dfrac{\text{length}(\gamma_{p,q})}{2}\right) = [\alpha_p^+,\alpha_p^-;w \cdot \alpha_q^+, w \cdot \alpha_q^-]
\end{equation}
where $\alpha_j^{+}$ and $\alpha_j^{-}$ are respectively the attracting and repelling fixed points of $\rho_0(\alpha_j)$ on $S^1_{\infty}$, and $[z_1,z_2;z_3,z_4] \defeq \dfrac{(z_1-z_3)(z_2-z_4)}{(z_1-z_4)(z_2-z_3)}$ is the cross ratio.

Hence, Basmajian's identity \ref{bas_torus} becomes
\begin{equation} \label{alg}
\sum_{j=1}^k l(\rho(\alpha_j)) = \sum_{p,q =1}^k \sum_{w \in \mathscr{L}_{p,q}}\log [\alpha_p^+,\alpha_p^-;w\cdot \alpha_q^+,w \cdot \alpha_q^-]
\end{equation}

\subsection{Complexified Basmajian's identity}
When we start deforming a Fuchsian marking, identity \ref{alg} can be analytically continued only when the series is absolutely convergent.

\begin{thm} \label{thm_cxBas_conv}
	Given a marked Schottky representation $\rho:F_n \to \text{PSL}(2,\bbC)$, the series in \ref*{alg} converges absolutely if and only if the Hausdorff dimension of the limit set $\Lambda_{\Gamma}$ of the Schottky group $\Gamma=\rho(F_n)$ is strictly less than one.
\end{thm}

The proof of the theorem uses the Thermodynamic Formalism and will be presented in section 6, after we introduce the necessary technical tools in section 5.

\begin{prop} \label{cor_unif_conv}
	Let $\rho_t:F_n \to PSL(2,\bbC)$ be a deformation of a Fuchsian marking $\rho_0$ in $\mathcal{S}_{<1}$. Then the series in \ref*{alg} converges uniformly on compact subsets for $t\in [0,1]$.
\end{prop}
\begin{proof}
	The proposition follows from the proof of Theorem \ref*{thm_cxBas_conv}. Since the absolute series is uniformly bounded above by the geometric series $\sum_{k=1}^{\infty} \lambda_1^k$, which is continuous, it is uniformly bounded.
\end{proof}

We are now ready to prove the main theorem of this section.

\begin{thm} [Complexified Basmajian's identity] \label{thm_cx_Bas}
	Suppose $\rho_0: F_n \to \text{PSL}(2, \bbC)$ is a Fuchsian marking corresponding to a hyperbolic surface $\Sigma$ with geodesic boundaries $a_1,\cdots, a_k$. Let $\alpha_1,\cdots, \alpha_k \in \pi_1M$ represent the free homotopy classes of $a_1,\cdots, a_k$. If $\rho$ is in the same path component as $\rho_0$ in $\mathcal{S}_{<1}$, then
	\begin{equation}\label{cx_Bas_id}
	\sum_{j=1}^k l(\rho(\alpha_j)) = \sum_{p,q =1}^k \sum_{w \in \mathcal{L}_{p,q}} \log [\alpha_p^+,\alpha_p^-;w\cdot \alpha_q^+,w \cdot \alpha_q^-] \text{~~mod~~} 2\pi i
	\end{equation}
	where $\alpha_j^+,\alpha_j^-$ are the attracting and repelling fixed points of $\rho(\alpha_j)$, respectively. Moreover, the series converges absolutely.
\end{thm}

\begin{proof}
	Let $\rho_t$ in $\mathcal{S}_{<1}$ be an analytic deformation of $\rho_0$. By Proposition \ref{cor_unif_conv}, each side of identity \ref*{cx_Bas_id} is a holomorphic function on $\mathcal{S}_{<1}$ (up to $2\pi i$). In a neighbourhood $U$ of $\rho_0$, there exist local coordinates such that the Fuchsian markings are a totally real subspace of $U$. Therefore, the identity holds true for the entire neighbourhood $U$. Hence, by analytic continuation, the identity holds for all $\rho_t$ (up to $2\pi i$).
\end{proof}

\subsection{Monodromy}
In this subsection, we exhibit two different loops in $\mathcal{S}_{<1}$ along which the identity exhibits nontrivial monodromy. As we analytically continue \ref*{cx_Bas_id} along a loop in $\mathcal{S}_{<1}$ the value of the right hand side might change by multiples of $2\pi i$. We call this the {\it monodromy} of the loop and observe that it depends only on its homotopy class in $\mathcal{S}_{<1}$. Note that the uniform convergence throughout a compact loop in $\mathcal{S}_{<1}$ implies that only finitely many words can change monodromy.

Let $F_2 = \langle a,b \rangle$ be a free group on two generators with $A = a^{-1}$ and $B = b^{-1}$. Let $L,x \in \bbC$ be such that $x+1/x = -L$. Consider the Schottky groups $\Gamma_L \defeq \rho(\langle a,b \rangle) = \langle X_L, Y_L \rangle$ and $\Gamma'_L \defeq \rho'(\langle a,b \rangle) = \langle X_L^2, X_LY_L^3 \rangle$, where
$$X_L = \left[ \begin{array}{cc}
L & 1 \\
-1 & 0 \end{array} \right] \text{ , } 
Y_L = \left[ \begin{array}{cc}
0 & x \\
-\frac{1}{x} & L \end{array} \right].$$
Now let $L_t = 5e^{2\pi it}$. Then as $t$ increases from $0$ to $1$, the trajectory of $\Gamma_{L_t}$ forms a loop $\gamma$ in the $\text{PSL}(2,\bbC)-$character variety of $F_2$ which connects the Fuchsian group $\Gamma_{L_0}$ to itself. Similarly, the trajectory of $\Gamma'_{L_t}$ defines another loop $\gamma'$.

Numerical calculations suggest that both $\gamma$ and $\gamma'$ are in fact in $\mathcal{S}_{<1}$ as the series are absolutely convergent. The homotopy class of these loop in $\mathcal{S}_{<1}$ is distinguished by their monodromy in Table 1.

\begin{table}[h] \label{table_mon}
\begin{center}
	\begin{tabular}{ |c|c|c| } 
		\hline
		Words & Monodromy along $\gamma$\text{~~} & Monodromy along $\gamma'$ \\ 
		\hline
		a & $2\pi$ & $10\pi$\\ 
		\hline
		b & $2\pi$ & $6\pi$\\  
		\hline
		A & $2\pi$ & $10\pi$\\ 
		\hline 
		B & $2\pi$ & $6\pi$\\
		\hline
		ab & $2\pi$ & $0$\\  
		\hline 
		AB & $2\pi$ & $0$\\  
		\hline 
		aB & $0$ & $2\pi$\\  
		\hline
		Ab & $0$ & $2\pi$\\  
		\hline
		Total change & $12\pi$ & $36\pi$\\  
		\hline  
	\end{tabular}
	\vspace*{0.2cm}
\caption{Monodromy for the loops $\gamma$ and $\gamma'$}
\end{center}
\end{table}

\section{Quadratic Polynomials} \label{sec_Julia}
The main goal of this section is to introduce a Basmajian-type identity on Cantor Julia sets $J_c$ of complex quadratic polynomials $f_c(z) = z^2+c$ with $c$ lying outside of the Mandelbrot set $\mathcal{M}$. First of all, the series identity is obtained for $c < -2$ as in this case $J_c$ is a cut-out set in $\bbR$. Similar to the case of Schotkky groups, we state that the series is absolutely convergent if and only if the Hausdorff dimension of $J_c$ is strictly less than one, which enables us to analytically continue the identity. When we analytically continue the series along a nontrivial loop, the monodromy induces a new Basmajian-type identity on $J_c$. In addition, using the convergence of the series, we numerically plot the Hausdorff dimension one locus in $\bbC \setminus \mathcal{M}$.

\subsection{Basmajian-type identity}
Fix $c<-2$. Let $$T_1 = \sqrt{z-c} \text{ and } T_2 = -\sqrt{z-c}$$ be local branches of $f_c^{-1}$ and let $z_1, z_2$ be the fixed points of $T_1, T_2$ respectively. Then the Julia set $J_c$ is obtained as a cut-out set contained in the interval $U=[-z_1, z_1]$ as follows. Set $$U_0 \defeq U$$ For each $n \ge 1$, let 
$$U_n = T_1(U_{n-1}) \cup T_2(U_{n-1}) \text{ and  } I_n = U \setminus U_n.$$ Then $U_n$ consists of $2^n$ intervals whose lengths shrink with $n$ and in the limit $U_{\infty} = J_c$ is a Cantor set of zero measure.

\begin{lem}(Properties of gaps)
	\begin{enumerate}
		\item $I_1 = [T_2(-z_1),T_1(-z_1)]$ is the largest gap in $U$.
		\item Images of $I_1$ under $T_1$ and $T_2$ give rise to all the gaps (i.e. $I_{\infty}$) in $U$.
		\item The lengths of all the gaps add up to the length of $U$.
	\end{enumerate}
\end{lem}
\begin{proof}
	(1) follows from the monotonicity of $T_1$ and $T_2$ and the fact that they are contractions.
	(2) follows directly from the definitions.
	Since $U$ is the smallest interval containing $J_c$ and the Cantor set has measure zero, (3) follows.  
\end{proof}

By part (3), we have the following identity
\begin{equation} \label{Julia_id}
z_1 - (-z_1) = \sum_{w \in \{T_1,T_2\}^*} (-1)^\eta \Big( w(T_1(-z_1)) - w(T_2(-z_1)) \Big)
\end{equation}
where $\eta$ is the number of $T_2$'s in the word $w$.

For a general parameter $c \in \bbC \setminus \mathcal{M}$, we will show in section \ref{proof_conv_thm} that
\begin{thm} \label{thm_Julia_conv}
	The series in \ref{Julia_id} is absolutely convergent if and only if the Hausdorff dimension of $J_c$ is strictly less than one.
\end{thm}

Again, analyticity of Hausdorff dimension (Corollary \ref{cor_ana_hdim}) implies that the subspace $\left(\bbC \setminus \mathcal{M}\right)_{<1} \defeq \{c \notin \mathcal{M} ~|~ \text{dim}_{\text{H}} J_c < 1 \}$ is open.

Hence, by analytic continuation, we obtain

\begin{thm} \label{prop_Julia_id}
	For complex parameter $c \in \left(\bbC \setminus \mathcal{M}\right)_{<1}$, let $T_1$ and $T_2$ be the two branches of $f_c^{-1}$ and $z_1$ be the fixed point of $T_1$, then the following identity holds
	\begin{equation} 
	z_1 - (-z_1) = \sum_{w \in \{T_1,T_2\}^*} (-1)^\eta \Big( w(T_1(-z_1)) - w(T_2(-z_1)) \Big)
	\end{equation}
	where $\eta$ is the number of $T_2$'s in the word $w$.
\end{thm}

\subsection{Monodromy}
In subsection 5.4 we will see that the Julia set $J_c$ can be symbolically coded as the set of strings
$$\Sigma = \big\{ \underline i = (i_0, i_1, i_2,\cdots) ~|~ i_j \in  \{1,2\} \text{ for all } j \in \bbN \big\}.$$

The monodromy group respecting the dynamics can then be identified with $\text{Aut}(\Sigma)$ which is $\bbZ/2\bbZ$ (\cite{Hed}). Hence, the nontrivial monodromy $\phi: J_c \to J_c$ simply exchanges the labels $1$ and $2$ in the symbolic coding.

This map defines a new identity on $J_c$, since the original identity would still hold on $\phi(J_c) = J_c$ after being continued along a loop in $\left(\bbC \setminus \mathcal{M}\right)_{<1}$ starting and ending at $c$. Thus,
\begin{equation} 
z_2 - (-z_2) = \sum_{w \in \{T_1,T_2\}^*} (-1)^{\eta'} \Big( w(T_2(-z_2)) - w(T_1(-z_2)) \Big)
\end{equation}
where $\eta'$ is the number of $T_1$'s in the word $w$.

\subsection{Hausdorff dimension one locus}
Outside of the Mandelbrot set, by analyticity of Hausdorff dimension, the Hausdorff dimension one locus $S(1)$ is a closed analytic set and therefore it has to intersect every ray emanating from the origin. For each fixed ray, we numerically find the points at which the convergence of the series of absolute values changes. Our numerical results as shown in Figure 2 suggest that $S(1)$ is a topological circle connecting $-2$ to itself with a cusp at $-2$. Figure 3 is a zoomed in picture near $-2$. This gives numerical evidence for the following:
\begin{conj}
$\left(\bbC \setminus \mathcal{M}\right)_{>1}$ is star-shaped centered at $0$.
\end{conj}

\begin{figure}[h!]
\centering
\label{fig_hdim1}
\includegraphics[scale=0.6]{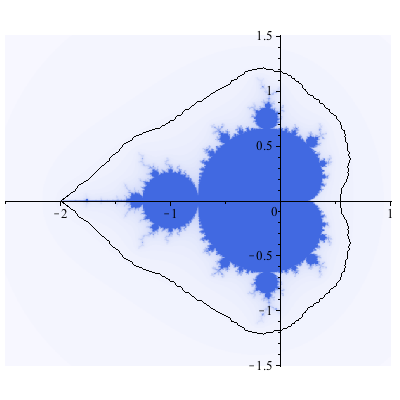}
\caption{Hausdorff dimension one locus in $\bbC \setminus \mathcal{M}$}
\end{figure}

\begin{figure}[h!]
\centering
\label{fig_hdim2}
\includegraphics[scale=0.6]{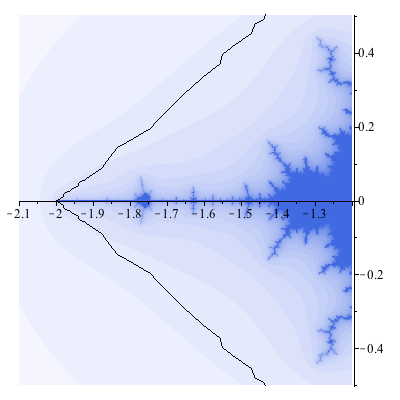}
\caption{Near $-2$}
\end{figure}
\newpage

\section{Thermodynamic Formalism} \label{sec_thermo}
Thermodynamic formalism is a powerful tool to study expanding conformal dynamical systems, especially to estimate Hausdorff dimensions of limit sets. As a generalization of Moran's theorem to nonlinear systems, Bowen (\cite{Bowen}) established a formula for computing Hausdorff dimension from the {\it pressure} function, which is also related to the maximum eigenvalue of the {\it Ruelle transfer operator}. 

In this section, we give a brief overview of Thermodynamic Formalism, following \cite{Pol} (see also \cite{Rue}). All theorems presented here are classical but we include some simple proofs for the convenience of the reader. In the end, we discuss its applications to Schottky groups and quadratic polynomials.

Throughout this section, $T : \Lambda \to \Lambda$ is a $C^{1+\alpha}$ locally expanding conformal map. The (local) branches of $T^{-1}$ are contractions which form an IFS and $\Lambda \subset \bbR^d$ is the limit set.

\subsection{Symbolic coding}
$T : \Lambda \to \Lambda$ is said to be a {\it Markov map} with respect to a {\it Markov partition} if there exists a finite collection $\mathcal{P}$ of closed subsets $\{P_i\}_{i=1}^m$ such that the following conditions are satisfied:
\begin{enumerate}
	\item $\Lambda = \bigcup_{i=1}^m P_i$;
	\item For each $i = 1,\cdots, m$, $P_i$ is the closure of its interior;
	\item For each $i = 1,\cdots, m$, $TP_i$ is a finite union of sets in $\mathcal{P}$.
\end{enumerate}
$\mathcal{P} = \{P_i\}_{i=1}^m$ is called a {\it Markov partition} for $\Lambda$.

Now suppose $T$ is an expanding conformal Markov map and its inverse $T^{-1}$ has $n$ (local) branches. Then the dynamics of the IFS consisting of branches of $T^{-1}$ can be symbolically coded.

Let $A$ be an $n \times n$ matrix such that
$$A_{i,j} =  \begin{cases} 1 &\mbox{if } T(P_i) \supset P_j \\ 
0 & \mbox{otherwise }  \end{cases} $$

Since $T$ is expanding,  $A$ is {\it aperiodic}, i.e. there exists an integer $N>0$ such that all the entries of $A^N$ are positive. For such a matrix $A$, define (one-sided) {\it subshift of finite type} 
$$\Sigma = \{\underline{i} = (i_j)_{j=0}^{\infty} ~|~ i_j \in \{1,2,\cdots, n\}, A_{i_j,i_{j+1}}=1\}.$$

Endow a metric on $\Sigma$ $$d(\underline{x},\underline{y}) = \dfrac{1}{2^n}$$ where $n$ is the first place $\underline{x}$ and $\underline{y}$ differ. With the topology induced by $d$, $\Sigma$ is a Cantor set. 

There is a H\"{o}lder continuous projection 
$$\pi: \Sigma \to X$$ such that $\underline{i} = (i_0, i_1, i_2\cdots)$ gives the sequence of sets $P_{i_0}, P_{i_1}, P_{i_2} \cdots$ visited by the forward orbit of $\pi(\underline{i})$.

The {\it shift map} $\sigma: \Sigma \to \Sigma$ defined by 
$$\sigma(i_0,i_1,i_2,\cdots) = (i_1, i_2, i_3,\cdots)$$ is locally expanding with respect to the metric $d$. The symbolic dynamical system $(\Sigma,\sigma)$ is conjugate to $(\Lambda, T)$ and is called a {\it symbolic coding}.

\subsection{Transfer operator, pressure and Hausdorff dimension}
In this subsection, we establish Bowen's formula and deduce that Hausdorff dimension is analytic for an analytic perturbation of $T$.

Throughout the subsection, we adopt the following notation. For each $\underline{i} = (i_0, i_1,\cdots) \in \Sigma$, write $T_{\underline i} = f_{i_n}\circ \cdots \circ f_{i_1} $, where $f_1, \cdots ,f_n$ are the local inverse branches of $T$, and $P_{\underline i} = f_{i_n}\circ \cdots \circ f_{i_1}P_{i_0}$. Denote $|P_{\underline i}|$ the diameter of the set $P_{\underline i}$.

In order to find Hausdorff dimension of $\Lambda$, we can cover $\Lambda$ by (open neighbourhoods of) $P_{\underline i}$ with $|P_{\underline i}| < \delta$. Then the $t$-dimensional Hausdorff measure of $\Lambda$ is estimated by $\sum_{|\underline i| = n} |P_{\underline i}|^t$, for large $n$ and $0< t < d$. 

\begin{prop}[\cite{Pol}, Proposition 2.4.1] \label{estimate}
	~~
	\begin{enumerate}
		\item There are positive constants $B_1$ and $B_2$ such that for any $\underline{i}$ and $x,y \in \Lambda$, $$B_1 \le \dfrac{|T_{\underline{i}}'(x)|}{|T_{\underline{i}}'(y)|} \le B_2.$$
		\item There exist constants $C_1$ and $C_2$ such that for any $\underline i$ and $z \in \Lambda$, $$C_1 \le \dfrac{|P_{\underline i}|}{|T_{\underline i}'(z)|} \le C_2.$$ In particular, there are constants $C_1$ and $C_2$ such that for any $n \ge 1$ and $z \in \Lambda$, $$C_1 \le \dfrac{\sum_{|\underline i| = n}|P_{\underline i}|^t}{\sum_{|\underline i| = n}|T_{\underline i}'(z)|^t} \le C_2.$$
	\end{enumerate}
\end{prop}

Therefore, we see that estimating $\sum_{|\underline i| = n} |P_{\underline i}|^t$ is equivalent to estimating $\sum_{|\underline i| = n}|T_{\underline i}'(z)|^t$. However, the latter quantity has a great advantage as it is related to the {\it transfer operator}, which we define now.

For $\alpha>0$, let $C^{\alpha}(\Sigma)$ be the Banach space of real-valued H\"{o}lder continuous functions $f : \Sigma \to \bbR$, i.e. $$|f(\underline{x})-f(\underline{y})| \le Cd(\underline{x},\underline{y})^{\alpha}$$
with H\"{o}lder norm given by 
$$||f||_{\alpha} = \sup_{\underline{x}}|f(\underline{x})| + \sup_{\underline{x} \neq \underline{y}}\dfrac{|f(\underline{x})-f(\underline{y})|}{d(\underline{x},\underline{y})^{\alpha}}$$

\begin{defn} [Transfer operator]
	For any H\"{o}lder continuous potential $\phi \in C^{\alpha}(\Sigma)$, define the {\it transfer operator} $\mathcal{L}_{\phi}: C^{\alpha}(\Sigma) \to C^{\alpha}(\Sigma)$ as $$\big(\mathcal{L}_{\phi}\psi\big)(\underline y) = \sum_{\sigma(\underline x) = \underline y}e^{\phi(\underline x)}\psi(\underline x).$$
\end{defn}

To establish Bowen's formula, we consider the following one-parameter family of H\"{o}lder potentials
$$\phi_t = -t\log|T'(\pi(\underline x))|, \text{ with } 0<t<d,$$
and the corresponding transfer operators $\mathcal{L}_t \defeq \mathcal{L}_{\phi_t}$. By straightforward computation, $\mathcal{L}_t: C^{\alpha}(\Lambda) \to C^{\alpha}(\Lambda)$ is given by
$$\big(\mathcal{L}_t\psi\big)(z) = \sum_{j=1}^{n} |f_j'(x)|^t\psi(f_j(z)).$$

Again, straightfoward computation shows that the $n$-th iterate of $\mathcal{L}_t$ evaluated at the constant function $\mathds{1}$ gives the desired quantity, i.e. 
$$\big(\mathcal{L}_t^n \mathds{1}\big)(z) = \sum_{|\underline i| = n}|T_{\underline i}'(z)|^t, \text{ for any } z \in \Lambda.$$

The following celebrated Ruelle-Perron-Frobenius theorem will allow us to see that the Hausdorff dimension of $\Lambda$ is in fact related to the spectrum of $\mathcal{L}_t$.  

\begin{thm}[Ruelle-Perron-Frobenius Theorem, \cite{Pol}, Proposition 2.4.2] \label{thm_RPF}
~~
\begin{enumerate}
	\item $\mathcal{L}_t$ has a simple maximum eigenvalue $\lambda_t >0$ and there is exponential convergence $|\mathcal{L}_t^n \mathds{1} - \lambda_t^n| \le C\lambda_t^n\theta^n$, for some $C > 1$, $0<\theta <1$ and $n \ge 1$.
	\item There exists a probability measure $\mu$ and $D_1$,$D_2>0$ such that for any $n \ge 1$ and $|\underline{i}|=n$ and $x \in \Lambda$
	$$D_1\lambda_t^n \le \dfrac{\mu(P_{\underline{i}})}{|T_{\underline{i}}'(x)|^t}\le D_2\lambda_t^n.$$
	\item The map $\lambda(t) = \lambda_t$ is real analytic and $\lambda'(t) <0$ for all $t \in \bbR$.
\end{enumerate}
\end{thm}

Finally, we are in a position to state Bowen's formula for computing Hausdorff dimension.

\begin{defn} \label{def_pressure}
	For any continuous function $f: \Lambda \to \bbR$, define {\it pressure} $P(f)$ to be
	$$P(f) = \limsup_{n \to \infty} \dfrac{1}{n}\log \Big(\sum_{T^nx=x}e^{f(x)+f(Tx)+ \cdots +f(T^{n-1}x)}\Big).$$
\end{defn}

\begin{thm} [Bowen \cite{Bowen}] \label{Bowen_Ruelle}
	Let $T: \Lambda \to \Lambda$ be a $C^{1+\alpha}$ locally expanding map. Then the Hausdorff dimension $t$ of $\Lambda$ is the unique solution to $P(-t\log|T'|) = 0$.
\end{thm}

\begin{cor} \label{pressure_cor}
	$P(-t\log|T'|) = \displaystyle\limsup_{n \to \infty}\dfrac{1}{n}\log \left(\sum_{T^nx=x}\dfrac{1}{|(T^n)'(x)|^t}\right) = \log \lambda_t$.
\end{cor}

\begin{proof}
	The first equation follows from Definition \ref{def_pressure}.
	
	From Proposition \ref{estimate}, it follows that
	$$\sum_{T^nx=x}\dfrac{1}{|(T^n)'(x)|^t} \asymp \mathcal{L}_t^n \mathds{1} (x'), \text{ for any }x' \in \Lambda,$$	
	and by Theorem \ref{thm_RPF}, $\mathcal{L}_t^n \mathds{1} (x')\asymp \lambda_t^n$.
\end{proof}

\begin{cor} \label{analyticity_eigenvalue}
	The function $t \mapsto P(-t\log|T'|)$ is strictly monotone decreasing and analytic. 
\end{cor}
\begin{proof}
	It is an immediate consequence of Theorem \ref{thm_RPF} (3) and Corollary \ref{pressure_cor}.
\end{proof}

\begin{rmk} \label{lambda_one}
	$\lambda_s = 1$ when $s = $ Hausdorff dimension of $\Lambda$ and $\lambda_1 \le 1$ if and only if $s \le 1$.
\end{rmk}

\begin{cor}[Analyticity of Hausdorff dimension]\label{cor_ana_hdim}
	Let $T_{\lambda}$ be an analytic family of $C^{1+\alpha}$ locally expanding maps. Then the Hausdorff dimension as a function $\lambda \mapsto H.dim(\Lambda_{\lambda})$ is real analytic.
\end{cor}

\begin{proof}
	By Corollary \ref{analyticity_eigenvalue}, $f(t, \lambda) = P(-t\log|T_{\lambda}|)$ is analytic in $t$. Also, $\dfrac{\partial f(t, \lambda)}{\partial \lambda} \neq 0$. The corollary then follows from the inverse function theorem. 
\end{proof}

\subsection{Application I: Limit set of Schottky groups}
Recall that a Schottky group $\Gamma$ on $n$ generators is the image $\rho(F_n)$ under a marked Schottky representation $\rho: F_n \to \text{PSL}(2,\bbC)$. By definition, there are $n$ pairs of disjoint Jordan curves $C_i$, $C_i'$ on the boundary sphere $\hat{\bbC}$ such that $\rho(g_i)$ takes the exterior of $C_i'$ into the interior of $C_i$ and $\rho(g_i^{-1})$ takes the exterior of $C_i$ into the interior of $C_i'$. Then the Schottky group action is a conformal IFS on $\hat{\bbC}$ whose limit set $\Lambda_{\Gamma}$ is contained in the interior of these $2n$ Jordan curves.

The dynamics of $\Gamma$ has a natural symbolic coding. Set 
$$P_i \defeq \text{closed disk bounded by }C_i, \text{ and}$$ 
$$P_{2n+1-i} \defeq \text{closed disk bounded by }C_i'$$ for each $i = 1,\cdots,n$. Define $T: \Lambda_{\Gamma} \to \Lambda_{\Gamma}$ by 
$$T(z) = \begin{cases} \rho(g_i^{-1})(z) &\mbox{if } z \in P_i \\ 
\rho(g_i)(z) & \mbox{if } z \in P_{2n+1-i} \end{cases} \text{~for~} i =1,\cdots, n.$$ 
Then $T$ is a Markov map with respect to the Markov partition $\mathcal{P} = \{P_i\}_{i=1}^{2n}$. Furthermore, $T$ is expanding on $\Lambda_{\Gamma}$ (see \cite{Bowen} for details).

By definition,
\[ A = \left[\begin{array}{cccc}
1 & 1 & 0 & 1 \\
1 & 1 & 1 & 0 \\
0 & 1 & 1 & 1 \\
1 & 0 & 1 & 1 \\ \end{array} \right]\] 
and the subshift of finite type is the set $$\Sigma = \big\{ \underline i = (i_0, i_1, i_2,\cdots) ~|~ i_j \in \{g_1,\cdots, g_n, g_1^{-1},\cdots, g_n^{-1}\}, A_{i_j,i_k}=1\},$$ which is exactly the set of reduced words in the alphabet $\{g_1^{\pm 1},\cdots, g_n^{\pm 1}\}$.

\subsection{Application II: Julia set of quadratic polynomials}
For $c$ lying outside of the Mandelbrot set, the quadratic polynomial $f_c(z) = z^2 + c$ is expanding on its Cantor Julia set $J_c$. In fact, there exists a Markov partition $\{P_1, P_2\}$ with respect to which $f_c$ is a Markov map and satisfies $P_1 \subset f_c(P_1)$, $P_2 \subset f_c(P_1)$, $P_1 \subset f_c(P_2)$, $P_2 \subset f_c(P_2)$. 

Then, the matrix 
\[ A = \left[\begin{array}{cc}
1 & 1 \\
1 & 1 \\ \end{array} \right]\] 
defines the subshift of finite type $$\Sigma = \big\{ \underline i = (i_0, i_1, i_2,\cdots) ~|~ i_j \in  \{1,2\} \text{ for all } j \in \bbN \big\}.$$

With these symbolic codings, the mechanism of Thermodynamic Formalism applies to Schottky groups and quadratic polynomials, which will be used in the next section to prove the convergence theorems.

\section{Proof of the Convergence Theorems}\label{proof_conv_thm}

\subsection{Quadratic polynomials: proof of Theorem \ref{thm_Julia_conv}}
The proof consists of the following two lemmas.
\begin{lem} If the Hausdorff dimension of the Julia set $J_c$ with $c \notin \mathcal{M}$ is less than one, then the series \ref{Julia_id} converges absolutely.
\end{lem}

\begin{proof}
	Since the Hausdorff dimension of $J_c <1$, by Remark \ref{lambda_one}, $\lambda_1 < 1$. If a word $w$ is identified with $\underline i$ by the symbolic coding, then image $w(I)$ of the largest gap $I$ under $w$ is contained in $P_{\underline i}$. Thus, summing up all the words up to length $N$, we have
	$$\sum_{|w| = 1}^N |w(I)| \le  \sum_{|\underline i|= 1}^N |P_{\underline i}| \le C \displaystyle\sum_{n=1}^N \lambda_1^{n}, $$ for some constant $C$, since $\sum_{|\underline i|= n}|P_{\underline i}| \asymp \lambda_1^{n}$. 
	
	\noindent Letting $N \to \infty$ gives the desired result.
\end{proof}

\begin{lem} \label{lem_Julia_2}
	There exists a constant $C$ such that for all $n \ge 1$, we have
	$$\displaystyle\sum_{|\underline{i}|=n}|P_{\underline{i}}|-\sum_{|\underline{i}|=n+1}|P_{\underline{i}}| \le  C\displaystyle\sum_{|w|=n}|w(I)|.$$ 
	In particular, if the series is absolutely convergent, then the Hausdorff dimension of the Julia set is less than one.
\end{lem}

\begin{proof}
	Fix $n = |\underline{i}|$. Denote $P_{\underline i}^1$ and $P_{\underline i}^2$ the two disjoint sets contained in $P_{\underline i}$.
	Using estimates from Proposition \ref{estimate},
	$$\dfrac{C_1}{C_2}B_1^n|f_1'| \le \dfrac{|P_{\underline i}^1|}{|P_{\underline i}|} \le \dfrac{C_2}{C_1}B_2^n|f_1'|.$$ 
	Since $|f_1'|$ is bounded, there are constants $D_1$ and $D_2$ such that $$D_1 \le \dfrac{|P_{\underline i}^1|}{|P_{\underline i}|} \le D_2.$$
	
	The same holds for $P_{\underline i}^2$. Then,  
	$$|P_{\underline i}| - |P_{\underline i}^1|-|P_{\underline i}^2| \le |P_{\underline i}| (1-2D_1).$$

	On the other hand, $\dfrac{|P_{\underline i}|}{|w(I)|} \asymp \dfrac{|P_j|}{|f_j(I)|}$ for $j=1,2$. Then there are constants $D_3$ and $D_4$ such that for any word $w$ identified with $\underline i$ via the symbolic coding,
	$$D_4 \le \dfrac{|P_{\underline i}|}{|w(I)|} \le D_3$$
	
	Therefore, summing up all $\underline i$ and $w$ of length $n$, we have
	\begin{align*}
	\displaystyle\sum_{|\underline{i}|=n}|P_{\underline{i}}|-\sum_{|\underline{i}|=n+1}|P_{\underline{i}}| & \le  (1-2D_1)\displaystyle\sum_{|\underline{i}|=n}|P_{\underline i}| \\
	& \le (1-2D_1)D_3\displaystyle\sum_{|w|=n}|w(I)|.
	\end{align*}

	The right hand side of the inequality is finite whenever the series is absolute convergent.
	Thus $\lambda_1^n (1-\lambda_1) < \infty$. Since $n$ is arbitrary, $\lambda_1 < 1$.
\end{proof}

\subsection{Schottky groups: proof of Theorem \ref{thm_cxBas_conv}}

We begin with a special case which serves as a model for the general case.
\subsubsection{Torus with one boundary}
Let $\rho_0 : F_2 \to \text{PSL}(2,\bbC)$ be a Fuchsian marking whose underlying surface is a torus with a geodesic boundary $a_1$. $\pi_1M = F_2 = \langle a,b \rangle$ with an ordered symmetric generating set $S = \{a>b>A>B\}$, where capital letters denote inverses.

Normalize $\rho_0$ by conjugation so that $\rho([a,b])$ has $\infty$ and $0$ as attracting and repelling fixed point, respectively.

Then the series at this Fuchsian marking $\rho_0$ is 
$$ \sum_{w \in \mathscr{L}} \log [\infty,0 ;w\cdot \infty,w \cdot 0] =  \sum_{w \in \mathscr{L}} \log \dfrac{w(0)}{w(\infty)}$$ where $\mathscr{L}$ was specified in Example \ref{ex_torus}.

The following three lemmas constitute the proof of Theorem \ref{thm_cxBas_conv} for this special case.
\begin{lem}
	Take the principal branch of $\log$. Then the series $\displaystyle\sum_{w \in \mathscr{L}}\log \dfrac{w(0)}{w(\infty)}$ is absolutely convergent if and only if $\displaystyle\sum_{w \in \mathscr{L}} w(0) - w(\infty)$ is.
\end{lem}

\begin{proof}
	$\log(1+z) = \sum_{n=1}^{\infty}\dfrac{(-1)^{n+1}}{n}z^n$ for $|z|<1$. Then, for $|z|$ small, $|\log(1+z)| \asymp |z|$. Therefore, for $|w|$ large, 
	$$\biggr|\log \dfrac{w(0)}{w(\infty)}\biggr| = \biggr\rvert \log \left(1+ \dfrac{w(0)-w(\infty)}{w(\infty)}\right)\biggr\rvert \asymp \biggr| \dfrac{w(0)-w(\infty)}{w(\infty)}\biggr|\asymp | w(0)-w(\infty)\rvert$$ as the limit set is compact.
\end{proof}

\begin{lem} \label{lem_torus_1}
	The series $$\displaystyle\sum_{w \in \mathscr{L}} w(0) - w(\infty)$$ is absolutely convergent if the Hausdorff dimension of the limit set is less than one.
\end{lem}
\begin{proof}
	Suppose the Hausdorff dimension is less than one. Note that for $|w|=n>3$, $w(0), w(\infty) \in P_{\underline{i}}$ where first $n-3$ letters of $w$ and $\underline{i}$ are identified, i.e. $|w(0)-w(\infty)| \le |P_{\underline{i}}|$.
	Hence, $$\sum_{|w| = 4}^{N} |w(0)-w(\infty)| \le \sum_{|\underline{i}|= 1}^N|P_{\underline{i}}| \le C\sum_{k=1}^{N}\lambda_1^k,$$ for some constant $C$ as $\sum_{|\underline i|= n}|P_{\underline i}| \asymp \lambda_1^{n}$. The conclusion follows by letting $N \to \infty$. 
\end{proof}

\begin{lem} \label{lem_torus_2}
	If the series $$\displaystyle\sum_{w \in \mathscr{L}} w(0) - w(\infty)$$ is absolutely convergent, then the Hausdorff dimension of the limit set is less than one.
\end{lem}
\begin{proof}
	For each $n$ large, if $w$ of length $n$ is identified with $\underline{i}$, denote $P_{\underline i}^1$, $P_{\underline i}^2$ and $P_{\underline i}^3$ the three disjoint sets contained in $P_{\underline i}$. Then, exactly the same argument as in the proof of Lemma \ref{lem_Julia_2} shows that there exists a constant $C$ such that for all $\underline i$,  
	$$|P_{\underline i}| - |P_{\underline i}^1|-|P_{\underline i}^2|-|P_{\underline i}^3| \le C|w(0)-w(\infty)|.$$
	Summing all the words $w$ of length $n$,
	$$\sum_{\underline{i} \sim w}|P_{\underline i}| - \sum_{|\underline{i}|=n+1}|P_{\underline i}| \le C\sum_{|w|=n}|w(0)-w(\infty)|.$$
	
	Recall that $w$ does not start with $abA$ or $ba$, for $|\underline{i}| = m$, $4\cdot3^{m-3}$ of $4 \cdot 3^{m-1}$ $P_{\underline{i}}$'s are not visited by $w(0)$ or $w(\infty)$. Hence, $$\sum_{\underline{i} \sim w}|P_{\underline i}| = \dfrac{8}{9} \sum_{|\underline{i}|=n}|P_{\underline i}|=\dfrac{8}{9}\lambda_1^n.$$
	
	For $n$ large, we have
	$$\infty > \displaystyle\sum_{|w|=n}|w(0)-w(\infty)| \ge \dfrac{8}{9}\displaystyle\sum_{|\underline{i}|=n}|P_{\underline{i}}|-\sum_{|\underline{i}|=n+1}|P_{\underline{i}}| = \lambda_1^n (\frac{8}{9}-\lambda_1).$$
	Hence, $\lambda_1 < 1$. 
\end{proof}

\subsubsection{The general case}

\begin{lem}
	The series 
	\begin{equation}\label{series_Kleinian}
	\sum_{p,q =1}^k \sum_{w \in \mathscr{L}_{p,q}}\log [\alpha_p^+,\alpha_p^-;w\cdot \alpha_q^+,w \cdot \alpha_q^-]
	\end{equation} is absolutely convergent if the Hausdorff dimension of the limit set is less than one.
\end{lem}
\begin{proof}
	For each $p,q = 1,\cdots, k$, the absolute convergence of the series $$\displaystyle\sum_{w \in \mathscr{L}_{p,q}}\log [\alpha_p^+,\alpha_p^-;w\cdot \alpha_q^+,w \cdot \alpha_q^-]$$ is reduced to the torus case (possibly with different starting and ending conditions for RedLex $w$) and follows from Lemma \ref{lem_torus_1}.
\end{proof}

\begin{lem}
	If the series \ref{series_Kleinian} is absolutely convergent, then the Hausdorff dimension of the limit set is less than one.
\end{lem}
\begin{proof}
	In particular, for $p=q=1$, the series $$\displaystyle\sum_{w \in \mathscr{L}_{1,1}}\log [\alpha_1^+,\alpha_1^-;w\cdot \alpha_1^+,w \cdot \alpha_1^-]$$ is absolutely convergent. Then apply the same argument as in the proof of Lemma \ref{lem_torus_2}.
\end{proof}


\begin{thebibliography}{8}
\bibitem{Bas} A. Basmajian, {\it The orthogonal spectrum of a hyperbolic manifold}, American Journal of Mathematics, {\bf 115} (1993), no. 5, 1139–1159


\bibitem{Bers}L. Bers, {\it Automorphic forms for Schottky groups}, Advances in Mathematics, {\bf 16} (1975), 332–361

\bibitem{Bowen} R. Bowen, {\it Hausdorff dimension of quasi-circles}, IHES Publ. Math. {\bf 50} (1979), 1–25

\bibitem{C-K-W} D. Calegari, S. Koch and A. Walker, {\it Roots, Schottky semigroups, and a proof of Bandt's Conjecture}, arXiv:1410.8542


\bibitem{Fal}
K. Falconer, \emph{Techniques in Fractal Geometry}. John Wiley, 1997

\bibitem{Hed} G. Hedlund, {\it Endomorphisms and automorphisms of the shift dynamical system}, Math. Syst. Theory {\bf 3} (1969)



\bibitem{Pol} M. Pollicott. {\it Lectures on fractals and dimension theory}, retrieved from http://homepages.warwick.ac.uk/~masdbl/dimension-total.pdf

\bibitem{Rue} D. Ruelle, {\it Thermodynamic formalism}, Addison-Wesley Publishing Co., Reading, Massachusetts, 1978





\end{thebibliography}
\end{document}